\documentclass[a4paper,oneside,reqno]{amsart}

\usepackage[utf8]{inputenc}
\usepackage{textgreek}
\usepackage[colorlinks,unicode]{hyperref}
\usepackage{xspace}
\usepackage{float}
\usepackage{appendix}
\usepackage{tikz}
\usepackage{pgfplots}\pgfplotsset{compat=1.18}
\usepackage{amsfonts}
\usepackage{amssymb}
\usepackage{mathtools}
\usepackage{braket}
\usepackage{mathtools}
\usepackage{siunitx}
\usepackage{enumitem}
\usepackage{algorithm}
\usepackage[noend]{algpseudocode}
\usepackage{array}
\usepackage{booktabs}
\usepackage{cleveref}
\usepackage{comment}
\usepackage{orcidlink}
\usepackage{doi}
\usepackage{color}

\hypersetup{
    linkcolor=blue
    ,citecolor=blue
    ,urlcolor=blue
}

\usetikzlibrary{cd}

\setlist[enumerate,1]{label={(\arabic*)}}

\makeatletter
\newcounter{algorithmicH}
\let\oldalgorithmic\algorithmic
\renewcommand{\algorithmic}{%
    \stepcounter{algorithmicH}
    \oldalgorithmic}
\renewcommand{\theHALG@line}{ALG@line.\thealgorithmicH.\arabic{ALG@line}}
\makeatother

\allowdisplaybreaks

\theoremstyle{plain}
\newtheorem{thm}{Theorem}[section]
\newtheorem{lem}[thm]{Lemma}
\newtheorem{cor}[thm]{Corollary}
\newtheorem{prop}[thm]{Proposition}

\newtheorem{defn}[thm]{Definition}
\theoremstyle{definition}
\newtheorem{rem}[thm]{Remark}

\theoremstyle{plain}

\crefname{thm}{theorem}{theorems}
\Crefname{thm}{Theorem}{Theorems}
\crefname{defn}{definition}{definitions}
\Crefname{defn}{Definition}{Definitions}
\crefname{prop}{proposition}{propositions}
\Crefname{prop}{Proposition}{Propositions}
\crefname{lem}{lemma}{lemmas}
\Crefname{lem}{Lemma}{Lemmas}
\crefname{cor}{corollary}{corollaries}
\Crefname{cor}{Corollary}{Corollaries}
\crefname{ex}{example}{examples}
\Crefname{ex}{Example}{Examples}
\crefname{rem}{remark}{remarks}
\Crefname{rem}{Remark}{Remarks}
\crefname{hyp}{hypothesis}{hypotheses}
\Crefname{hyp}{Hypothesis}{Hypotheses}

\makeatletter
\newcommand{\subalign}[1]{%
    \vcenter{%
        \Let@ \restore@math@cr \default@tag
        \baselineskip\fontdimen10 \scriptfont\tw@
        \advance\baselineskip\fontdimen12 \scriptfont\tw@
        \lineskip\thr@@\fontdimen8 \scriptfont\thr@@
        \lineskiplimit\lineskip
        \ialign{\hfil$\m@th\scriptstyle##$&$\m@th\scriptstyle{}##$\hfil\crcr
            #1\crcr
        }%
    }%
}
\makeatother

\newcommand{\F}{\mathbb{F}}                                         
\newcommand{\Fp}{\F_{p}}                                            
\newcommand{\Fq}{\F_{q}}                                            
\newcommand{\Fpn}{\Fp^{n}}                                          
\newcommand{\Fqn}{\Fq^{n}}                                          
\newcommand{\Fqnxn}{\Fq^{n \times n}}                               
\newcommand{\Fqm}{\Fq^{m}}                                          
\newcommand{\Fqx}{\Fq^\times}                                       

\DeclareMathOperator{\Hom}{Hom}                                     
\DeclareMathOperator{\Tr}{Tr}                                       
\ifdefined\widebar
\else
    \newcommand\widebar[1]{\mathop{\overline{#1}}}                  
\fi

\newcommand{\AffSp}[2]{\mathbb{A}^{#1}_{#2}}                        
\newcommand{\ProjSp}[2]{\mathbb{P}^{#1}_{#2}}                       

\newcommand{\homog}{\text{\normalfont hom}}                         

\DeclareMathOperator{\prob}{\mathbb{P}}                             

\newcommand*{\degree}[1]{\deg \left( #1 \right)}                    

\DeclareMathOperator{\Spec}{Spec}                                   

\DeclareMathOperator{\wt}{wt}                                       

\newcommand{\Anemoi}{\texttt{Anemoi}\xspace}

\newcommand{\Arion}{\textsf{Arion}\xspace}

\newcommand{\Ciminion}{\texttt{Ciminion}\xspace}

\newcommand{\GMiMC}{\texttt{GMiMC}\xspace}

\newcommand{\Griffin}{\textsc{Griffin}\xspace}

\newcommand{\Hades}{\textsc{Hades}\xspace}

\newcommand{\LowMC}{\texttt{LowMC}\xspace}

\newcommand{\MiMC}{\texttt{MiMC}\xspace}

\newcommand{\Poseidon}{\textsc{Poseidon}\xspace}
\newcommand{\Poseidontwo}{\textsc{Poseidon}2\xspace}

\newcommand{\ReinforcedConcrete}{\texttt{Reinforced Concrete}\xspace}

\begin{document}

    \title[]{A Degree Bound For The c-Boomerang Uniformity Of Permutation Monomials}

    \author[M.\ J.\ Steiner]{Matthias Johann Steiner}
    \address{Alpen-Adria-Universit\"at Klagenfurt, Universit\"atsstraße 65-67, 9020 Klagenfurt am W\"orthersee, Austria}
    \email{matthias.steiner@aau.at}

    \thanks{Matthias Steiner has been supported in part by the KWF under project number KWF-$3520 \vert 31870 \vert 45842$ and by the European Research Council (ERC) under the European Union's Horizon 2020 research and innovation program (grant agreement No.\ 725042).}

    \begin{abstract}
        Let $\mathbb{F}_q$ be a finite field of characteristic $p$.
        In this paper we prove that the $c$-Boomerang Uniformity, $c \neq 0$, for all permutation monomials $x^d$, where $d > 1$ and $p \nmid d$, is bounded by $d^2$.
        Further, we utilize this bound to estimate the $c$-boomerang uniformity of a large class of Generalized Triangular Dynamical Systems, a polynomial-based approach to describe cryptographic permutations, including the well-known Substitution-Permutation Network.
    \end{abstract}


    \maketitle

    \section{Introduction}
    Differential cryptanalysis \cite{C:BihSha90} introduced by Biham \& Shamir is one of the most successful attack vectors against symmetric key ciphers and hash functions.
    It is based on the propagation of input-output differences through the rounds of an iterated construction.
    A generalization of the differential attack is the boomerang attack \cite{FSE:Wagner99} introduced by Wagner.
    In a boomerang attack one splits an iterated function into two parts and then combines two differentials for the upper and lower part of the construction.
    Cid et al.\ introduced the \emph{Boomerang Connectivity Table} (BCT) \cite{EC:CHPSS18}  analog to the \emph{Differential Distribution Table} (DDT) to quantify the vulnerability of a construction against boomerang attacks.
    Utilizing the notion of the BCT Boura \& Canteaut introduced the \emph{boomerang uniformity} \cite{ToSC:BouCan18}, that is the maximum entry of the BCT for non-trivial input-output differences, analog to the differential uniformity to express the susceptibility of S-boxes against boomerang attacks.
    Finally, Ellingsen et al.\ generalized the DDT by considering $c$-scaled differences \cite{Ellingsen-cDifferential} and St\u{a}nic\u{a} utilized the same approach to generalize the BCT to $c$-scaled differences \cite{Stanica-cBoomerang}.
    Articles in this line of research usually come with tons of estimations for the $c$-differential and $c$-boomerang uniformity of various permutations.
    In a simplifying manner we classify these permutations into the following categories:
    \begin{itemize}
        \item Variants of the inverse permutation $x^{-1}$.

        \item Quadratic permutations over binary fields.

        \item Special permutations, e.g.\ Gold functions.

        \item Special permutations over special fields.
    \end{itemize}
    For an excerpt of these investigations we refer to \cite{Mesnager-Boomerang,ToSC:BouCan18,Hasan-Boomerang,Stanica-Characterization,Stanica-SwappedInverse,Stanica-WeilSums}, also recently Mesnager et al.\ \cite{Mesnager-Survey} performed a survey on trends in generalized differential and boomerang cryptanalysis.

    Unfortunately, these bounds for the $c$-BCT are unsuitable to evaluate the resistance of ciphers and hash functions needed for the efficient implementation of \emph{Multi-Party Computation} (MPC) and \emph{Zero-Knowledge} (ZK) proof systems.
    Designs for MPC \& ZK have a different performance measure than classical bit based constructions.
    Most MPC \& ZK protocols are defined over a prime field $\Fp$, where $p \geq 2^{60}$, and for efficient implementation ciphers and hash functions should be native in $\Fp$.
    Moreover, they require a low number of multiplications for evaluation.
    In the literature, such designs are also known as \emph{Arithmetization-Oriented} (AO) designs.
    Examples for AO ciphers and hash functions are \LowMC \cite{EC:ARSTZ15}, \MiMC \cite{AC:AGRRT16}, \GMiMC \cite{ESORICS:AGPRRRRS19}, \Hades \cite{EC:GLRRS20}, \Poseidon \cite{USENIX:GKRRS21}, \Poseidontwo \cite{Poseidon2}, \Ciminion \cite{EC:DGGK21}, \ReinforcedConcrete \cite{CCS:GKLRSW22}, \Griffin \cite{Griffin}, \Anemoi \cite{Anemoi} and \Arion \cite{Arion}.
    At round level, most of these primitives utilize low degree power permutations $x^d$, where $d \in \{ 3, 5, 7 \}$.
    Surprisingly, the boomerang uniformity of $x^d$ over large prime fields $\Fp$ is unknown yet.

    Suppose we are given a univariate cipher, like \MiMC, over $\Fp$ that applies the power permutation $x^d$ in every round, i.e.\
    \begin{equation}
        \begin{split}
            \mathcal{R}^{(i)} (x, k) &= (x + k + c_i )^d, \\
            \mathcal{C} (x, k) &= \mathcal{R}^{(r)} \circ \cdots \circ \mathcal{R}^{(1)} (x, k) + k,
        \end{split}
    \end{equation}
    where $k$ denotes the key and composition is taken with respect to the plaintext variable $x$.
    A boomerang differential covering two rounds of $\mathcal{C}$ is the connection of a differential of $\mathcal{R}^{(i)} (x, k)$ and a differential of ${\mathcal{R}^{(i + 1)}}^{-1} (x, k)$.
    It is well-known that a permutation and its inverse share the same differential uniformity, see \cite[Proposition~1]{EC:Nyberg93}, also over $\Fp$ it is easy to see that the differential uniformity of $x^d$ is bounded by $d$.
    So by intuition we expect that the boomerang uniformity of two rounds of $\mathcal{C}$ is bounded by $d^2$.
    In this paper we provide a formal proof for this intuition, in particular we prove the following bound for the $c$-boomerang uniformity of power permutations.
    \begin{thm}[\Cref{Cor: boomerang uniformity of monomials}]
        Let $\Fq$ be a finite field of characteristic $p$, let $c \in \Fqx$, and let $d \in \mathbb{Z}_{> 1}$ be such that $\gcd (d, q - 1) = 1$ and $p \nmid d$.
        Then
        \begin{enumerate}
            \item $\beta_{x^d, c} \leq d^2$.

            \item $\beta_{x^\frac{1}{d}, c} \leq d^2$.
        \end{enumerate}
    \end{thm}

    St\u{a}nic\u{a} has shown that for any permutation $F$ over $\Fq$ the $c$-boomerang uniformity can be computed by finding the number of solutions of two differentials for $F$, see \cite[Theorem~4]{Stanica-Characterization}.
    We can view these differentials as two polynomials in two variables, so each polynomial defines a plane curve in $\ProjSp{2}{\Fq}$ after homogenization.
    Therefore, the number of joint solutions is simply upper bounded by the intersection number of these two plane curves.
    Provided these curves do not have any common irreducible components, then B\'ezout's theorem asserts that the intersection number is simply the product of the degrees of the curves, see \cite[Theorem~5.61]{Goertz-AlgGeom}.
    For the requirement of B\'ezout's theorem, we will prove that one of the curves is irreducible and that the other curve is not a multiple of the irreducible one.

    In the last part of the paper we will study $c$-uniformities of the \emph{Generalized Triangular Dynamical System} (GTDS) \cite{GTDS}.
    The GTDS is a generic polynomial framework that unifies the most prominent block cipher design strategies, being the Substitution-Permutation Network and the Feistel Network, in a single primitive.
    For large classes of GTDS we will prove bounds on the $c$-differential uniformity as well as the $c$-boomerang uniformity.

    \section{Preliminaries}
    Let $k$ be a field, then we denote with $\bar{k}$ the algebraic closure of $k$.
    We denote the multiplicative group of a field with $k^\times = k \setminus \{ 0 \}$.
    Elements of a vector space $\mathbf{a} \in k^n$ are denoted with small bold letters, and matrices $\mathbf{M} \in k^{m \times n}$ are denoted by capital bold letters.
    Moreover, we denote by $\mathbf{M} \mathbf{a}$ the matrix-vector product.

    Let $f \in k [x_1, \dots, x_n]$ be a polynomial, and let $x_0$ be an additional variable.
    We call
    \begin{equation}
        f^\homog (x_0, \dots, x_n) = x_0^{\degree{f}} \cdot f \left( \frac{x_1}{x_0}, \dots, \frac{x_n}{x_0} \right) \in k [x_0, \dots, x_n]
    \end{equation}
    the homogenization of $f$ with respect to $x_0$.

    By $\Fq$ we denote the finite field with $q$ elements.
    It is well-known that for $d \in \mathbb{Z}_{\geq 1}$ $x^d$ induces a permutation over $\Fq$ if and only if $\gcd (d, q - 1) = 1$, see \cite[7.8.~Theorem]{Niederreiter-FiniteFields}.
    If not specified otherwise we denote the inverse of $x^d$ by $x^\frac{1}{d}$, where $\frac{1}{d}$ represents the unique integer $1 \leq e < q - 1$ such that $e \cdot d \equiv 1 \mod (q - 1)$.

    \subsection{c-Differential Uniformity}
    Key quantity to measure the capabilities of differential cryptanalysis is the so-called differential uniformity \cite{EC:Nyberg93} of a function.
    Ellingsen et al.\ \cite{Ellingsen-cDifferential} generalized this concept by admitting $c$-scaled differences.
    \begin{defn}[{c-differential uniformity, \cite[Definition~1]{Ellingsen-cDifferential}}]
        Let $\Fq$ be a finite field, let $c \in \Fqx$, and let $F: \Fqn \to \Fqm$ be a function.
        \begin{enumerate}
            \item Let $\boldsymbol{\Delta x} \in \Fqn$ and $\boldsymbol{\Delta y} \in \Fqm$.
            The entry of the $c$-Differential Distribution Table ($c$-DDT) of $F$ at $(\boldsymbol{\Delta x}, \boldsymbol{\Delta y})$ is defined as
            \begin{equation*}
                {}_{c}\delta_F (\boldsymbol{\Delta x}, \boldsymbol{\Delta y}) = \left| \left\{ \mathbf{x} \in \Fqn \mid F (\mathbf{x} + \boldsymbol{\Delta x}) - c \cdot F (\mathbf{x}) = \boldsymbol{\Delta y} \right\} \right|.
            \end{equation*}

            \item The $c$-differential uniformity of $F$ is defined as
            \begin{equation*}
                {}_{c} \delta (F) = \max \left\{ {}_{c}\delta_F (\boldsymbol{\Delta x}, \boldsymbol{\Delta y}) \mid \boldsymbol{\Delta x} \in \Fqn,\ \boldsymbol{\Delta y} \in \Fqm,\ \boldsymbol{\Delta x} \neq \mathbf{0} \ \text{if} \ c = 1 \right\}.
            \end{equation*}
        \end{enumerate}
    \end{defn}

    Obviously, for $c = 1$ we obtain the usual notion of differential uniformity.
    We recall some properties of the $c$-differential uniformity.
    \begin{prop}[{\cite[Proposition~5]{Stanica-cBoomerang}}]\label[prop]{Prop: c-DDT inverse}
        Let $\Fq$ be a finite field, let $c \in \Fqx$, and let $d \in \mathbb{Z}_{> 1}$ be such that $\gcd (d, q - 1) = 1$.
        Then ${}_{c} \delta_{x^\frac{1}{d}} (\Delta x, \Delta y) = {}_{c^{-d}} \delta_{x^d} \Big( c^{-1} \cdot \Delta y, c^{-d} \cdot \Delta x \Big)$.
    \end{prop}

    For power functions, calculating the $c$-differential uniformity is straight-forward when the characteristic does not divide the exponent.
    \begin{lem}
        Let $\Fq$ be a finite field of characteristic $p$, let $c \in \Fqx$, and let $d \in \mathbb{Z}_{> 1}$ be such that $p \nmid d$.
        Then ${}_{c} \delta \big( x^d \big) \leq d$.
    \end{lem}
    \begin{proof}
        Let $\Delta x \in \Fqx$ and $\Delta y \in \Fq$.
        For $c \neq 1$, $(x + \Delta x)^d - c \cdot x^d = \Delta y$ is a non-trivial polynomial of degree $d$, so it as $\leq d$ many roots in $\Fq$.
        For $c = 1$, $\sum_{i = 0}^{d - 1} \binom{d}{i} \cdot (\Delta x)^{d - i} \cdot x^i = \Delta y$ has the leading term $d \cdot \Delta x$ since $p \nmid d$, so it is a non-trivial polynomial of degree $d - 1$ which has lass than $d$ roots in $\Fq$.
    \end{proof}

    \subsection{c-Boomerang Uniformity}
    Now let us recall the definition of the $c$-boomer\-ang uniformity.
    \begin{defn}[{c-boomerang uniformity, \cite[Remark~1]{Stanica-cBoomerang}}]\label[defn]{Def: boomerang uniformity}
        Let $\Fq$ be a finite field, let $c \in \Fq^\times$, and let $F: \Fqn \to \Fqm$ be a function.
        \begin{enumerate}
            \item\label{Item: boomerang connectivity table} Let $\mathbf{a} \in \Fqn$ and $\mathbf{b} \in \Fqm$.
            The entry of the $c$-Boomerang Connectivity Table ($c$-BCT) of $F$ at $(\mathbf{a}, \mathbf{b})$ is defined as
            \begin{equation*}
                {}_{c}B_F (\mathbf{a}, \mathbf{b}) = \left| \left\{ (\mathbf{x}, \mathbf{y}) \in \Fqn \times \Fqn \; \middle\vert \;
                \begin{aligned}
                    F (\mathbf{x} + \mathbf{y}) &- c \cdot F (\mathbf{x}) = \mathbf{b} \\
                    c \cdot F (\mathbf{x} + \mathbf{y} + \mathbf{a}) &- F (\mathbf{x} + \mathbf{a}) = c \cdot \mathbf{b}
                \end{aligned}
                \right\}
                \right|.
            \end{equation*}

            \item The $c$-boomerang uniformity of $F$ is defined as
            \begin{equation*}
                \beta_{F, c} = \max_{\substack{\mathbf{a} \in \Fqn \setminus \{ \mathbf{0 }\}, \\ \mathbf{b} \in \Fqm \setminus \{ \mathbf{0} \}}} {}_{c}B_F (\mathbf{a}, \mathbf{b}).
            \end{equation*}
        \end{enumerate}
    \end{defn}
    \begin{rem}\label[rem]{Rem: original definition}
        Let $F: \Fpn \to \Fpn$ be a permutation.
        St\u{a}nic\u{a} originally defined the $c$-BCT only for permutations \cite[Definition~3]{Stanica-cBoomerang}
        \begin{equation}\label{Equ: boomerang original definition}
            {}_{c}B_F (\mathbf{a}, \mathbf{b}) = \left| \left\{ \mathbf{x} \in \Fqn \; \middle\vert \; F^{-1} \left( c^{-1} \cdot F (\mathbf{x} + \mathbf{a}) + \mathbf{b} \right) - F^{-1} \big( c \cdot F (\mathbf{x}) + \mathbf{b} \big) = \mathbf{a} \right\} \right|.
        \end{equation}
        For $c = 1$ this coincides with the original BCT definition \cite[Definition~3.1]{EC:CHPSS18} over $\F_{2^n}$.
        Moreover, \Cref{Def: boomerang uniformity} \ref{Item: boomerang connectivity table} and \Cref{Equ: boomerang original definition} are equivalent for permutations \cite[Theorem~4]{Stanica-cBoomerang} but \Cref{Def: boomerang uniformity} \ref{Item: boomerang connectivity table} is more general since it can also be applied to non-permutations.
    \end{rem}

    Ideally, we would like to have a similar relationship for the $c$-BCT of $x^d$ and $x^\frac{1}{d}$ as in \Cref{Prop: c-DDT inverse} for the $c$-DDT.
    After some algebraic manipulations we indeed can derive two differential equations for $x^\frac{1}{d}$ starting from the $c$-boomerang equations for $x^d$, though only for $c = 1$ we again obtain a boomerang equation system.
    \begin{lem}\label[lem]{Lem: inverse relation}
        Let $\Fq$ be a finite field, let $c \in \Fqx$, and let $d \in \mathbb{Z}_{> 1}$ be such that $\gcd (d, q - 1) \allowbreak = 1$.
        Then
        \begin{equation*}
            {}_{c} \mathcal{B}_{x^d} (a, b) = \left| \left\{ (x, y) \in \Fq^2 \; \middle\vert \;
            \begin{aligned}
                (x + y)^\frac{1}{d} &- c^{-\frac{1}{d}} \cdot x^\frac{1}{d} = c^{-\frac{1}{d}} \cdot a \\
                c^{-\frac{1}{d}} \cdot (x + y + b)^\frac{1}{d} &- \left( x + \frac{b}{c} \right)^\frac{1}{d} = c^{-\frac{1}{d}} \cdot a
            \end{aligned}
            \right\} \right|.
        \end{equation*}
    \end{lem}
    \begin{proof}
        In the $c$-boomerang equations we first do the substitution $z = x + y$, then we rearrange the system
        \begin{align*}
            z^d &= c \cdot x^d + b, \\
            (x + a)^d &= c \cdot (z + a)^d - c \cdot b.
        \end{align*}
        Since $d$ induces a permutation over $\Fq$, for any $(x, z) \in \Fq^2$ this system of equations is equivalent to
        \begin{align*}
            z &= \left( c \cdot x^d + b \right)^\frac{1}{d}, \\
            x + a &= c^\frac{1}{d} \cdot \left( (z + a)^d - b \right)^\frac{1}{d}.
        \end{align*}
        Moreover, for $\Fq$-valued solutions we can perform the substitution $\hat{x} = x^\frac{1}{d}$ and $\tilde{z} = (z + a)^d$, this then yields the equivalent system of equations
        \begin{align*}
            \tilde{z}^\frac{1}{d} - a &= \left( c \cdot \hat{x} + b \right)^\frac{1}{d}, \\
            \hat{x}^\frac{1}{d} + a &= c^\frac{1}{d} \cdot \left( \tilde{z} - b \right)^\frac{1}{d}.
        \end{align*}
        Now we perform the substitution $\hat{z} = \tilde{z} - b$, divide the first equation by $c^{-\frac{1}{d}}$ and rearrange the equations
        \begin{align*}
            c^{-\frac{1}{d}} \cdot \left( \hat{z} +  b \right)^\frac{1}{d} - \left( \hat{x} + \frac{b}{c} \right)^\frac{1}{d} &= c^{-\frac{1}{d}} \cdot a, \\
            \hat{z}^\frac{1}{d} - c^{-\frac{1}{d}} \cdot \hat{x}^\frac{1}{d} &= c^{-\frac{1}{d}} \cdot a.
        \end{align*}
        After a final substitution $\hat{z} = \hat{x} + \hat{y}$ we obtain the claim.
    \end{proof}

    Let $\Fq$ be a finite field of characteristic $p$, an additive character of $\Fq$ is a map $\chi: \Fq \to \mathbb{C}$ such that $\chi (x + y) = \chi (x) \cdot \chi (y)$.
    Let $\Tr_{\Fq / \Fp}: \Fqn \to \Fp$ be the trace function, then
    \begin{equation}
        \chi_1 (x) = \exp \left( 2 \cdot i \cdot \pi \cdot \Tr_{\Fq / \Fp} (x) \right)
    \end{equation}
    is called the fundamental character of $\Fq$.
    It is well-known that any additive character $\chi$ of $\Fq$ is of the form $\chi (x) = \chi_1 (a \cdot x)$ for some $a \in \Fq$, see \cite[5.7.~Theorem]{Niederreiter-FiniteFields}.
    To the best of our knowledge, closest to an estimation of the $c$-boomerang uniformity of $x^d$ is the following characterization of the $c$-BCT in terms of character sums due to St\u{a}nic\u{a}.
    \begin{thm}[{\cite[Theorem~1]{Stanica-Characterization}}]
        Let $\Fq$ be a finite field, let $c \in \Fqx$, and let $d \in \mathbb{Z}_{\geq 1}$.
        For $a \in \Fqx$ and $b \in \Fq$ the $c$-Boomerang Connectivity Table entry ${}_{c} \mathcal{B}_{x^d} (a, a^d \cdot b)$ of $x^d$ is given by
        \begin{equation*}
            \frac{1}{q} \cdot \left( {}_{c} \delta_{x^d} (1, b) + {}_{c^{-1}} \delta_{x^d} (1, b) \right) + 1 + \frac{1}{q^2} \cdot \sum_{\substack{\alpha, \beta \in \Fq, \\ \alpha \cdot \beta \neq 0}} \chi_1 \big( -b \cdot (\alpha + \beta) \big) \cdot S_{\alpha, \beta} \cdot S_{-\alpha \cdot c, -\beta \cdot c^{-1}},
        \end{equation*}
        with
        \begin{equation*}
            S_{\alpha, \beta} = \sum_{x \in \Fq} \chi_1 \left( \alpha \cdot x^d \right) \cdot \chi_1 \left( \beta \cdot \left( x + 1 \right)^d \right).
        \end{equation*}
    \end{thm}

    Combined with \cite[Theorem~4]{Stanica-WeilSums} one could then compute entries of the $c$-BCT via character sums.
    Though, the AO designs mentioned in the introduction are all defined over prime fields $p \geq 2^{60}$, and over such fields these computations are simply infeasible in practice.

    Finally, let us compute the $c$-boomerang uniformity for a special class of polynomials.
    Over a finite field $\F_{p^n}$, a polynomial of the form
    \begin{equation*}
        f (x) = \sum_{i = 1}^{n - 1} a_i \cdot x^{p^i}
    \end{equation*}
    is called a \emph{linearized polynomial} since $f (x + y) = f (x) + f (y)$ for all $x, y \in \F_{p^n}$.
    \begin{lem}\label[lem]{Lem: boomerang uniformity of linearized polynomials}
        Let $\Fq$ be a finite field, let $c \in \Fqx$, and let $f \in \Fq [x]$ be a linearized permutation polynomial.
        Then
        \begin{equation*}
            {}_{c} \mathcal{B}_f (a, b) =
            \begin{cases}
                q, & c^2 = 1, \\
                1, & c^2 \neq 1.
            \end{cases}
        \end{equation*}
    \end{lem}
    \begin{proof}
        We have for the boomerang equation that
        \begin{align*}
            &f^{-1} \Big( c^{-1} \cdot f (x + a) + b \Big) - f^{-1} \Big( c \cdot f (x) + b \Big) \\
            &= f^{-1} \Big( c^{-1} \cdot f (x + a) \Big) - f^{-1} \Big( c \cdot f (x) \Big) \\
            &= f^{-1} \Big( c^{-1} \cdot f (x) \Big) - f^{-1} \Big( c \cdot f (x) \Big) + a \\
            &= a.
        \end{align*}
        Therefore,
        \begin{equation*}
            c^{-1} \cdot f(x) = c \cdot f(x) \quad \Longleftrightarrow \quad \left( c^{-1} - c \right) \cdot f(x) = 0.
        \end{equation*}
        If $c^2 = 1$, then this equation has $q$ solutions, and if $c^2 \neq 1$ this equation has an unique solution.
    \end{proof}

    Therefore, for the remainder of the paper we are only going to study non-linearized permutation monomials.

    \subsection{Plane Curves}
    Let $R$ be a ring, then $\Spec (R) = \{ \mathfrak{p} \subset R \mid \mathfrak{p} \text{ prime ideal} \}$ is called the spectrum of $R$.
    Also, $\Spec (R)$ can be equipped with the Zariski topology, see \cite[\S 2.1]{Goertz-AlgGeom}.
    A locally ringed space $(X, \mathcal{O}_X)$ is a topological space $X$ together with a sheaf of commutative rings $\mathcal{O}_X$ on $X$.
    If in addition $(X, \mathcal{O}_X)$ is isomorphic to $\big( \Spec (R), \mathcal{O}_{\Spec (R)} \big)$ for some ring $R$, then $(X, \mathcal{O}_X)$ is called an affine scheme.
    A scheme $(X, \mathcal{O}_X)$ is a locally ringed space that admits an open covering $X = \bigcup_i U_i$ of affine schemes $(U_i, \mathcal{O}_X \vert_{U_i})$.
    For a scheme $(X, \mathcal{O}_X)$ the ring $\Gamma (X, \mathcal{O}_X) = \mathcal{O}_X (X)$ is also known as the global sections of $X$.
    For a general introduction into the theory of schemes we refer to \cite{Goertz-AlgGeom}.

    Let $R$ be a ring, then the $n$-dimensional affine space $\AffSp{n}{R}$ is defined to be the scheme $\Spec \big( R [x_1, \dots, x_n] \big)$.
    Moreover, the $n$-dimensional projective space $\ProjSp{n}{R}$ over $R$ can be constructed by gluing $n + 1$ copies of the affine space $\AffSp{n}{R}$, see \cite[\S 3.6]{Goertz-AlgGeom}.
    Note that there also exists a natural morphism of schemes $p: \AffSp{n + 1}{R} \setminus \{ 0 \} \to \ProjSp{n}{R}$.
    Now let $R [X_0, \dots, X_n]$ be graded via the degree, and let $I \subset R [X_0, \dots, X_n]$ be a homogeneous ideal analog to the construction of $p$ one can construct a scheme $\mathcal{V}_+ (I)$ together with a closed immersion $\iota: \mathcal{V}_+ (I) \to \ProjSp{n}{R}$, $\mathcal{V}_+ (I)$ is also called the vanishing scheme of $I$, see \cite[3.7]{Goertz-AlgGeom}.

    Now let $k$ be a field, it is well-known that $\ProjSp{n}{k}$ and any closed subspace of it are separated and of finite type over $k$.
    Let $F, G \in k [X_0, X, Y]$ be homogeneous polynomials, then $\mathcal{V}_+ (F) \subset \ProjSp{2}{k}$ is called a plane curve over $k$.
    Moreover, we have for the scheme-theoretic intersection of two plain curves, see \cite[Example~4.38]{Goertz-AlgGeom},
    \begin{equation}
        \mathcal{V}_+ (F) \cap \mathcal{V}_+ (G) = \mathcal{V}_+ (F, G) \subset \ProjSp{2}{k}.
    \end{equation}

    A classical problem in algebraic geometry is to count the number of intersection points of two plane curves together with their multiplicities.
    \begin{defn}[{\cite[Definition~5.60]{Goertz-AlgGeom}}]
        Let $k$ be a field, and let $C, D \subset \ProjSp{2}{k}$ be two plane curves such that $Z := C \cap D$ is a $k$-scheme of dimension $0$.
        Then we call $i (C, D) := \dim_k \big( \Gamma (Z, \mathcal{O}_Z) \big)$ the intersection number of $C$ and $D$.
        For $z \in Z$ we call $i_z (C, D) := \dim_k (\mathcal{O}_{Z, z})$ the intersection number of $C$ and $D$ at $z$.
    \end{defn}

    It can be easily seen that
    \begin{equation}
        i (C, D) = \sum_{z \in C \cap D} i_z (C, D).
    \end{equation}

    By B\'ezout's well-known theorem, for two plane curves that do not have any common irreducible components the intersection number is simply the product of the degrees of the curves.
    \begin{thm}[{B\'ezout's theorem, \cite[Theorem~5.61]{Goertz-AlgGeom}}]
        Let $k$ be a field, and let $C = \mathcal{V}_+ (F)$ and $D = \mathcal{V}_+ (G)$ be plane curves in $\ProjSp{2}{k}$ given by polynomials without a common factor.
        Then
        \begin{equation*}
            i (C, D) = \degree{F} \cdot \degree{G}.
        \end{equation*}
        In particular, the intersection $C \cap D$ is non-empty and consists of a finite number of closed points.
    \end{thm}

    \section{Generalized c-Boomerang Equation For Monomials}
    To estimate the $c$-boomerang uniformity of $x^d$ we study the following system of equations in two variables.
    \begin{defn}
        Let $\Fq$ be a finite field, let $d \in \mathbb{Z}_{> 1}$, and let $c_1, \dots, c_5 \in \Fq$.
        Then we call
        \begin{equation*}
            \mathcal{F} (d, c_1, \dots, c_5) =
            \left\{
            \begin{array}{ c c c c }
                z^d & - & c_1 \cdot x^d & = c_2, \\
                (z + 1)^d & - & c_3 \cdot (x + c_4)^d & = c_5,
            \end{array}
            \right.
        \end{equation*}
        the generalized boomerang equation system for $x^d$.
    \end{defn}

    Note that we can transform any $c$-boomerang equation
    \begin{equation}\label{Equ: substitution}
        \begin{split}
            (x + y)^d &- c \cdot x^d = b \\
            c \cdot (x + y + a)^d &- (x + a)^d = c \cdot b
        \end{split}
    \end{equation}
    into this shape by dividing both equations by $a^d$, the second equation by $c$, and then performing the substitutions
    \begin{equation}\label{Equ: transformation}
        \begin{aligned}
            &\hat{z} = \frac{x + y}{a}, \qquad &&\hat{x} = \frac{x}{a}, \qquad &&c_1 = c, \\
            &c_2 = c_5 = \frac{b}{a^d}, \qquad &&c_3 = c^{-1},          \qquad &&c_4 = 1.
        \end{aligned}
    \end{equation}

    To bound the number of solutions of the generalized boomerang equation system we need two lemmata.
    \begin{lem}\label[lem]{Lem: permutation monomial factorization}
        Let $\Fq$ be a finite field of characteristic $p$, let $d \in \mathbb{Z}_{> 1}$ be such that $\gcd (d, q - 1) = 1$ and $p \nmid d$, and let $a, b \in \Fqx$ be such that $b^d = a$.
        Then $x^d - a$ factors as
        \begin{equation*}
            x^d - a = (x - b) \cdot g (x),
        \end{equation*}
        where $g \in \Fq [x]$ does not have any roots in $\Fq$.
    \end{lem}
    \begin{proof}
        For $a,b \neq 0$, we have the factorization
        \begin{equation*}
            x^d - a = (x - b) \cdot \left( \sum_{i = 0}^{d - 1} a_i \cdot x^i \right),
        \end{equation*}
        where
        \begin{align*}
            -b \cdot a_0 &= -a, \\
            a_0 - b \cdot a_1 &= 0, \\
            &\ldots \\
            a_{d - 2} - b \cdot a_{d - 1} &= 0, \\
            a_{d - 1} &= 1.
        \end{align*}
        Therefore, $a_i = \frac{a}{b^{i + 1}}$ and henceforth $g (x) = \frac{a}{b} \cdot \sum_{i = 0}^{d - 1} \left( \frac{x}{b} \right)^i$.
        Let us verify that $b$ is not a root of $g$:
        \begin{equation*}
            g (b) = \frac{a}{b} \cdot \sum_{i = 0}^{d - 1} \left( \frac{b}{b} \right)^i = \frac{a}{b} \cdot d \neq 0,
        \end{equation*}
        since the characteristic $p$ does not divide $d$.
        On the other hand, $x^d = a$ must have a unique solution in $\Fq$, see \cite[7.1.~Lemma]{Niederreiter-FiniteFields}, so $g$ cannot have any roots in $\Fq$.
    \end{proof}
    \begin{lem}\label[lem]{Lem: irreducibility}
        Let $\Fq$ be a finite field of characteristic $p$, let $d \in \mathbb{Z}_{> 1}$ be such that $\gcd \left( d, q - 1\right) = 1$ and $p \nmid d$, and let $a, b \in \Fqx$.
        Then
        \begin{equation*}
            z^d - a \cdot x^d - b
        \end{equation*}
        is irreducible over $\overline{\Fq}$.
    \end{lem}
    \begin{proof}
        Let $f = z^d - a \cdot x^d - b$ and let $R = \overline{\Fq} [x]$, we are going to apply Eisenstein's irreducibility criterion \cite[Chapter~IV~\S 3~Theorem~3.1]{Lang-Algebra} over $R [z]$.
        Then $f$ has the coefficients $a_d = 1$, $a_{d - 1}, \dots, a_1 = 0$ and $a_0 = a \cdot x^d - b$.
        By \Cref{Lem: permutation monomial factorization} we have the factorization
        \begin{equation*}
            a \cdot x^d + b = (x - c) \cdot g (x),
        \end{equation*}
        where $c \in \Fqx$ is the unique solution of $x^d = -\frac{b}{a}$ and $g \in \Fq [x]$ does not have any roots over $\Fq$.
        Now we choose the prime ideal $\mathfrak{p} = (x - c) \subset R$, then trivially we have that $a_0, \dots, a_{d - 1} \in \mathfrak{p}$ and $a_d \notin \mathfrak{p}$.
        Since $c$ cannot be a root of $g (x)$ in the algebraic closure $\overline{\Fq}$ (if it were, then $g (x)$ would have roots over $\Fq$) we also have that $a_0^2 \notin \mathfrak{p}$.
        So by Eisenstein's criterion $f$ is irreducible over $\Fq$ as well as $\overline{\Fq}$.
    \end{proof}

    Now we can use B\'ezout's theorem to upper bound the number of solutions of the generalized boomerang equation.
    \begin{thm}\label[thm]{Th: number of solutions of generalized boomerang equation}
        Let $\Fq$ be a finite field of characteristic $p$, let $d \in \mathbb{Z}_{> 1}$ be such that $\gcd (d, q - 1) = 1$ and $p \nmid d$, and let $c_1, \dots, c_5 \in \Fqx$.
        Then the generalized boomerang equation system $\mathcal{F} (d, c_1, \dots, c_5)$ has at most $d^2$ many solutions over the algebraic closure $\overline{\Fq}$.
    \end{thm}
    \begin{proof}
        For ease of reading we denote the algebraic closure of $\Fq$ as $k = \overline{\Fq}$.
        Let $\mathcal{F} (d, c_1, \dots,\allowbreak c_5) = \{ f_1, f_2 \} \subset k [z, x]$ be the generalized boomerang equation system.
        First we homogenize the polynomial system with respect to the variable $X_0$, and we denote the homogenizations as $F_1 = f_1^\homog$ and $F_2 = f_2^\homog$.
        Then $F_1$ and $F_2$ define plane curves over $k$, i.e.\ closed subschemes $C = \mathcal{V}_+ (F_1)$ and $D = \mathcal{V}_+ (F_2)$ of $\ProjSp{2}{k}$.
        Note that factorization of a polynomial is invariant under homogenization, see \cite[Proposition~4.3.2]{Kreuzer-CompAlg2}, so by \Cref{Lem: irreducibility} $F_1$ is irreducible in $k [X_0, Z, X]$.
        To apply B\'ezout's theorem we have to verify that $F_1$ and $F_2$ do not have a common factor, so let us take a look at the inhomogeneous equations
        \begin{align*}
            f_1 &= z^d  -  c_1 \cdot x^d - c_2, \\
            f_2 &= (z + 1)^d  -  c_3 \cdot (x + c_4)^d - c_5.
        \end{align*}
        Suppose they have a common factor, since $f_1$ is irreducible by \Cref{Lem: irreducibility} we can then only have that $f_2 = \alpha \cdot f_1$, where $\alpha \in k^\times$.
        On the other hand, since by assumption $d$ is not a power of the characteristic $p$ we have for some $1 \leq i \leq d - 1$ that $\binom{d}{i} \neq 0$ in $\Fq$.
        I.e., the monomial $z^i$ is present in $f_2$ and therefore $f_2$ cannot be a multiple of $f_1$.
        So, by B\'ezout's theorem \cite[Theorem~5.61]{Goertz-AlgGeom} we have for the intersection number of $C$ and $D$ that
        \begin{equation*}
            i \big( C, D \big) = \degree{F_1} \cdot \degree{F_2} = d^2.
        \end{equation*}
        On the other hand, by definition of the intersection number we have that
        \begin{align*}
            i \big( C, D \big)
            &= \sum_{z \in C \cap D} \dim_k \left( \mathcal{O}_{C \cap D, z} \right) \\
            &\geq \sum_{z \in C \cap D} 1,
        \end{align*}
        where the latter inequality follows from $k \hookrightarrow \mathcal{O}_{C \cap D, z}$.
        By \cite[Lemma~5.59]{Goertz-AlgGeom} $C \cap D$ is a zero-dimensional scheme, now let $x \in C \cap D$ be a point.
        Then there exists an affine open subscheme $U = \Spec (A) \subset C \cap D$ such that $x \in U$.
        Moreover, by \cite[Lemma~5.7]{Goertz-AlgGeom} $0 \leq \dim (U) \leq \dim (C \cap D) = 0$, hence $x$ corresponds to a maximal ideal $\mathfrak{m}_x \in U$.
        Since $C \cap D$ is a closed subscheme of $\ProjSp{2}{k}$ it is of finite type, therefore $U$ is also of finite type.
        By Hilbert's Nullstellensatz \cite[Theorem~1.17]{Goertz-AlgGeom} $k \subset A / \mathfrak{m}_x$ is a finite field extension, so by \cite[Proposition~3.33]{Goertz-AlgGeom} this already implies that $x$ is closed in $C \cap D$.
        In particular, every point in our intersection is closed.
        Further, recall the set of $k$-valued points, see \cite[Example~5.3]{Goertz-AlgGeom},
        \begin{align*}
            C\cap D (k)
            = \mathcal{V}_+ (F_1, F_2) \big( k \big)
            &= \Hom_{k} \Big( \Spec \big( k \big), \mathcal{V}_+ (F_1, F_2) \Big) \\
            &= \left\{ \mathbf{x} \in \ProjSp{2}{k} \big( k \big) \; \middle\vert \; F_1 (\mathbf{x}) = F_2 (\mathbf{x}) = 0 \right\}.
        \end{align*}
        Since $k$ is algebraically closed and $\mathcal{V}_+ (F_1, F_2)$ is of finite type there is a one-to-one bijection between the closed points of $\mathcal{V}_+ (F_1, F_2)$ and the $k$-valued points of $\mathcal{V}_+ (F_1, F_2)$, see \cite[Corollary~3.36]{Goertz-AlgGeom}.
        Summing up all these observations we have derived the inequality
        \begin{equation*}
            \left| \left\{ \mathbf{x} \in \ProjSp{2}{k} \; \middle\vert \; F_1 (\mathbf{x}) = F_2 (\mathbf{x}) = 0 \right\} \right| \leq d^2.
        \end{equation*}
        Finally, let $U_0 = \left\{ (X_0, Z, X) \in \ProjSp{2}{k} \; \middle\vert \; T \neq 0 \right\}$, then we can compute the affine solutions of $f_1 = f_2 = 0$ via, see \cite[Chapter~8~\S 2~Proposition~6]{Cox-Ideals},
        \begin{equation*}
            \left\{ (z, x) \in k^2 \; \middle\vert \; f_1 (z, x) = f_2 (z, x) = 0 \right\} = \left\{ \mathbf{x} \in \ProjSp{2}{k} \; \middle\vert \; F_1 (\mathbf{x}) = F_2 (\mathbf{x}) = 0 \right\} \cap U_0.
        \end{equation*}
        So also the number of affine solutions to the generalized boomerang equation system $\mathcal{F} (d, c_1, \dots, c_5)$ is bounded by $d^2$ over $k$.
    \end{proof}

    \begin{cor}\label[cor]{Cor: boomerang uniformity of monomials}
        Let $\Fq$ be a finite field of characteristic $p$, let $c \in \Fqx$, and let $d \in \mathbb{Z}_{> 1}$ be such that $\gcd (d, q - 1) = 1$ and $p \nmid d$.
        Then
        \begin{enumerate}
            \item $\beta_{x^d, c} \leq d^2$.

            \item $\beta_{x^\frac{1}{d}, c} \leq d^2$.
        \end{enumerate}
    \end{cor}
    \begin{proof}
        The first claim follows from \Cref{Th: number of solutions of generalized boomerang equation} and the transformation of the $c$-boomerang equations into a generalized boomerang equation system, see \Cref{Equ: transformation}.
        For the second claim, recall that by \Cref{Lem: inverse relation}
        \begin{equation*}
            {}_{c} \mathcal{B}_{x^\frac{1}{d}} (a, b) = \left| \left\{ (x, y) \in \Fq^2 \; \middle\vert \;
            \begin{aligned}
                (x + y)^d &- c^{-d} \cdot x^\frac{1}{d} = c^{-d} \cdot a \\
                c^{-d} \cdot (x + y + b)^d &- \left( x + \frac{b}{c} \right)^d = c^{-d} \cdot a
            \end{aligned}
            \right\} \right|.
        \end{equation*}
        We can transform these two equations into a generalized boomerang equation system by dividing both equations by $b^d$, the second equation by $c^{-d}$ and the substitutions $\hat{z} = \frac{x + y}{b}$ and $\hat{x} = \frac{x}{b}$.
        Then we can again apply \Cref{Th: number of solutions of generalized boomerang equation}.
    \end{proof}

    \begin{cor}
        Let $\F_{p^n}$ be a finite field, let $c \in \F_{p^n}^\times$, let $d \in \mathbb{Z}_{> 1}$ be such that $\gcd (d, p^n - 1) = 1$ and $p \nmid d$, and let $1 \leq i < n$.
        Then
        \begin{enumerate}
            \item $\beta_{x^{d \cdot p^i}, c} \leq d^2$.

            \item $\beta_{x^\frac{1}{d \cdot p^i}, c} \leq d^2$.
        \end{enumerate}
    \end{cor}
    \begin{proof}
        Note that $(d \cdot p^i) \cdot p^{n - i} \equiv d \cdot p^n \equiv d \mod (p^n - 1)$, which implies that $\left( x^{d \cdot p^i} \right)^{p^{n - i}} \equiv x^d \mod \left( x^{p^n} - x \right)$.
        For (1), we set up the $c$-boomerang equations for $x^{d \cdot p^i}$, then we raise them to the $p^{n - i}$-th power.
        Since $x^{p^{n - i}}$ is a linearized monomial and we only care about $\F_{p^n}$-valued solutions, we can transform the equations with our previous observation into $c$-boomerang equations for $x^d$.
        Then the claim follows from \Cref{Th: number of solutions of generalized boomerang equation}.
        For (2), we apply the same argument to the equations from \Cref{Lem: inverse relation}.
    \end{proof}

    \begin{rem}
        For an arbitrary permutation polynomial $F \in \Fq [x]$, if one can prove that the polynomial
        \begin{equation*}
            F (x + y) - c \cdot F (x) - b
        \end{equation*}
        is irreducible for over $\overline{\Fq}$ for all $b, c \in \Fqx$, and that the polynomial
        \begin{equation*}
            c \cdot F (x + y + a) - F (x + a) - c \cdot b
        \end{equation*}
        has at least one term that is not present in $F (x + y) - c \cdot F (x) - b$ for all $a, b, c \in \Fqx$, then one can also prove via B\'ezout's theorem that $\beta_{F, c} \leq \degree{F}^2$.
    \end{rem}

    \subsection{The Cases \texorpdfstring{$a = 0$ \& $b = 0$}{a = 0 \& b = 0}}
    For completeness, let us also compute the $c$-BCT table entry for a permutation monomial $x^d$ when either $a = 0$ or $b = 0$.
    \begin{lem}
        Let $\Fq$ be a finite field, let $d \in \mathbb{Z}_{> 1}$ be such that $\gcd \left( d, q - 1 \right) = 1$, and let $a, b, c \in \Fqx$.
        Then
        \begin{enumerate}
            \item ${}_{c} B_{x^d} (0, b) =
            \begin{dcases}
                1, & c^2 \neq 1, \\
                q, & c^2 = 1.
            \end{dcases}
            $

            \item ${}_{c} B_{x^d} (a, 0) =
            \begin{dcases}
                1, & c^2 \neq 1, \\
                0, & c = -1, \\
                q, & c = 1.
            \end{dcases}
            $
        \end{enumerate}
    \end{lem}
    \begin{proof}
        For this proof we work with the $c$-BCT definition from \Cref{Rem: original definition}.
        For (1), we have that
        \begin{align*}
            \left( c^{-1} \cdot x^d + b \right)^\frac{1}{d} - \left( c \cdot x^d + b \right)^\frac{1}{d} &= 0 \\
            \Rightarrow c^{-1} \cdot x^d + b &= c \cdot x^d + b \\
            \Rightarrow \left( c^{-1} - c \right) \cdot x &= 0.
        \end{align*}
        If $c^2 = 1$, then the left-hand side vanishes, so there are $q$ solutions for $x$.
        If $c^2 \neq 1$, then there is an unique solution for $x$.

        For (2), we have that
        \begin{align*}
            \left( c^{-1} \cdot (x + a)^d \right)^\frac{1}{d} - \left( c \cdot x^d \right)^\frac{1}{d} &= a \\
            c^{-\frac{1}{d}} \cdot (x + a) - c^\frac{1}{d} \cdot x &= a \\
            \Rightarrow \left( c^{-\frac{1}{d}} - c^\frac{1}{d} \right) \cdot x &= \left( 1 - c^{-\frac{1}{d}} \right) \cdot a.
        \end{align*}
        We have that $c^{-\frac{1}{d}} = c^\frac{1}{d}$ if and only if $c^2 = 1$, else there is an unique solution for $x$.
        If $c = 1$, then the left and the right-hand side vanish, so we have $q$ solutions for $x$.
        If $c = -1$, then the left-hand side vanishes, but the right one does not, so we have $0$ solutions for $x$.
    \end{proof}

    \section{c-Uniformities Of The Generalized Triangular Dynamical System}
    The \emph{Generalized triangular dynamical system} (GTDS) \cite{GTDS} is a polynomial-based approach to describe cryptographic permutations over finite fields.
    It unifies the dominant block cipher strategies, being the Substitution-Permutation Network (SPN) and the Feistel Network, in a common primitive.
    As discussed in \cite[\S 4]{GTDS} almost all proposed symmetric key block ciphers can be considered as GTDS-based ciphers.
    \begin{defn}[{Generalized triangular dynamical system, \cite[Definition~7]{GTDS}}]\label[defn]{Def: generalized triangular dynamical system}
        Let $\Fq$ be a finite field, and let $n \geq 1$.
        For $1 \leq i \leq n$, let $p_i \in \Fq[x]$ be permutation polynomials, and for $1 \leq i \leq n - 1$, let $g_i, h_i \in \Fq[x_{i + 1}, \dots, x_n]$ be polynomials such that the polynomials $g_i$ do not have zeros over $\Fq$.
        Then we define a generalized triangular dynamical system $\mathcal{F} = \{ f_1, \dots, f_n \}$ as follows
        \begin{equation*}
            \begin{split}
                f_1(x_1,\dots,x_n)		&= p_1(x_1) \cdot g_1(x_2,\dots,x_n) + h_1(x_2,\dots,x_n),  \\
                f_2(x_1,\dots,x_n)		&= p_2(x_2) \cdot g_2(x_3,\dots,x_n) + h_2(x_3,\dots,x_n),	\\
                & \dots											                                    \\
                f_{n-1}(x_1,\dots,x_n)	&= p_{n-1}(x_{n-1}) \cdot g_{n-1}(x_n) + h_{n-1}(x_n),	    \\
                f_n(x_1,\dots,x_n)		&= p_n(x_n).
            \end{split}
        \end{equation*}
    \end{defn}

    It is easy to see that if all $p_i$'s are univariate permutations and the $g_i$'s do not have any zeros over $\Fq$, then the GTDS is invertible, see \cite[Proposition~8]{GTDS}.

    \subsection{c-Differential Uniformity Of The Generalized Triangular Dynamical System}
    In \cite[Theorem~18]{GTDS} the differential uniformity of the GTDS was bounded in terms of the degrees of the univariate permutation polynomials $p_i$ and the Hamming weight of the input difference.
    We now generalize this result to $c \neq 1$.
    As preparation, we recall a lemma.
    \begin{lem}[{\cite[Lemma~17]{GTDS}}]\label[lem]{Lem: degree and differential uniformity}
        Let $\Fq$ be a finite field, and let $f \in \Fq [x] / (x^q - x)$.
        Then $\delta (f) < q$ if and only if $\deg \big( f(x + a) - f(x) \big) > 0$ for all $a \in \Fqx$.
        In particular, if $\delta (f) < q$ then $\delta (f) < \degree{f}$.
    \end{lem}
    \begin{thm}\label[thm]{Th: c-differential distribution of GTDS}
        Let $\Fq$ be a finite field, let $n \geq 1$ be an integer, let $c \in \Fq \setminus \{ 0, 1 \}$, and let $\mathcal{F}: \Fqn \to \Fqn$ be a GTDS.
        Let $p_1, \dots, p_n \in \Fq [x] / (x^q - x)$ be the univariate permutation polynomials of the GTDS $\mathcal{F}$ such that for every $i$ one has $\delta (p_i) < q$.
        Let $\boldsymbol{\Delta x}, \boldsymbol{\Delta y} \in \Fqn$.
        Then the $c$-Differential Distribution Table of $\mathcal{F}$ at $\boldsymbol{\Delta x}$ and $\boldsymbol{\Delta y}$ is bounded by
        \begin{equation*}
            {}_{c} \delta_\mathcal{F} ( \boldsymbol{\Delta x}, \boldsymbol{\Delta y} )
            \leq
            \degree{p_n} \cdot \prod_{i = 1}^{n - 1}
            \begin{dcases}
                \begin{rcases}
                    \degree{p_i}, & \boldsymbol{\Delta x}_i \neq 0, \\
                    q, & \boldsymbol{\Delta x}_i = 0
                \end{rcases}
            \end{dcases}
            .
        \end{equation*}
    \end{thm}
    \begin{proof}
        Suppose we are given the differential equation
        \begin{equation}\label{Equ: differential equation 1}
            \mathcal{F} (\mathbf{x} + \boldsymbol{\Delta x}) - c \cdot \mathcal{F} (\mathbf{x}) = \boldsymbol{\Delta y},
        \end{equation}
        Then, the last component of the differential equation only depends on the variable $x_n$, i.e.,
        \begin{equation*}
            p_n (x_n + \boldsymbol{\Delta x}_n) - c \cdot p_n (x_n) = \boldsymbol{\Delta y}_n.
        \end{equation*}
        This polynomial always has the leading monomial $(1 - c) \cdot x^{\degree{p_n}}$, therefore we can have at most $\degree{p_n}$ many solutions for $x_n$ in $\Fq$.

        Now suppose we have a solution for the last component, say $\hat{x}_n \in \Fq$.
        Then, we can substitute it in \Cref{Equ: differential equation 1} into the $(n - 1)$\textsuperscript{th} component
        \begin{equation*}
            f_{n - 1} (x_{n - 1} + \boldsymbol{\Delta x}_{n - 1}, \hat{x}_n + \boldsymbol{\Delta x}_n) - c \cdot f_{n - 1} (x_{n - 1}, \hat{x}_n) = \boldsymbol{\Delta y}_{n - 1}.
        \end{equation*}
        Since $\hat{x}_n$ is a field element we can reduce this equation to
        \begin{equation}\label{Equ: differential equation 2}
            \alpha \cdot p_{n - 1} (x_{n - 1} + \boldsymbol{\Delta x}_{n - 1} ) - c \cdot \beta \cdot p_{n - 1} (x_{n - 1}) + \gamma = \boldsymbol{\Delta y}_{n - 1},
        \end{equation}
        where $\alpha, \beta, \gamma \in \Fq$ and $\alpha, \beta \neq 0$.
        Now we have to do a case distinction on the various case for $\alpha$, $\beta$, and $\boldsymbol{\Delta x}_{n - 1}$.
        \begin{itemize}
            \item For $\alpha \neq c \cdot \beta$, we obtain a polynomial with leading monomial $(\alpha - c \cdot \beta) \cdot x^{\degree{p_{n - 1}}}$, therefore we can have at most $\degree{p_{n - 1}}$ many solutions for $x_{n - 1}$ in $\Fq$.

            \item For $\alpha = c \cdot \beta$, we obtain a rescaled differential equation for $p_{n - 1}$.
            \begin{itemize}
                \item If $\boldsymbol{\Delta x}_{n - 1} = 0$, then the equation becomes constant, so it can have at most $q$ many solutions for $x_{n - 1}$.

                \item If $\boldsymbol{\Delta x}_{n - 1} \neq 0$, then we obtain a proper differential equation for $p_{n - 1}$, so by \Cref{Lem: degree and differential uniformity} we can have at most $\degree{p_{n - 1}}$ many solutions for $x_{n - 1}$ in $\Fq$.
            \end{itemize}
        \end{itemize}
        Since in general we do not know which case for $\alpha$ and $\beta$ applies, we have to estimate the number of solutions for $\boldsymbol{\Delta x}_{n - 1} = 0$ by $q$, since we could end up with a constant equation.
        On the other, hand if $\boldsymbol{\Delta x}_{n - 1} \neq 0$ all possible cases are non-trivial and can have at most $\degree{p_{n - 1}}$ many solutions.

        Inductively, we now work upwards through the branches to derive the claim.
    \end{proof}

    Let $\wt: \Fqn \to \mathbb{Z}$ denote the Hamming weight of a vector, i.e.\ the number of non-zero entries of a vector.
    Moreover, let $\mathbf{a} \in \Fqn$, then we denote the restriction of $\mathbf{a}$ to its first $k$ entries by $\mathbf{a} \vert^k$.
    \begin{cor}\label{Cor: differential uniformity}
        Let $\Fq$ be a finite field, let $n \geq 1$ be an integer, let $c \in \Fq \setminus \{ 0, 1 \}$, and let $\mathcal{F}: \Fqn \to \Fqn$ be a GTDS.
        Let $p_1, \dots, p_n \in \Fq [x] / (x^q - x)$ be the univariate permutation polynomials of the GTDS $\mathcal{F}$, and let $\boldsymbol{\Delta x}, \boldsymbol{\Delta y} \in \Fqn$.
        If for all $1 \leq i \leq n$ one has that $1 < \degree{p_i} \leq d$ and $\delta (p_i) < q$, then
        \begin{equation*}
            {}_{c} \delta_\mathcal{F} (\boldsymbol{\Delta x}, \boldsymbol{\Delta y}) \leq d \cdot q^{n - 1 - \wt \big( \boldsymbol{\Delta x} \vert^{n - 1} \big)} \cdot d^{\wt \big( \boldsymbol{\Delta x} \vert^{n - 1} \big)}.
        \end{equation*}
        In particular,
        \begin{equation*}
            \prob \left[ \mathcal{F} \! : \boldsymbol{\Delta x} \to_{{}_{c} \delta} \boldsymbol{\Delta y} \right] \leq \frac{d}{q} \cdot \left( \frac{d}{q} \right)^{\wt \big( \boldsymbol{\Delta x} \vert^{n - 1} \big)}.
        \end{equation*}
    \end{cor}
    \begin{proof}
        The first claim follows from \Cref{Lem: degree and differential uniformity} and \Cref{Th: c-differential distribution of GTDS}, the bound for the probability follows from the first and division by $q^n$.
    \end{proof}

    \subsection{c-Boomerang Uniformity Of A Class Of Generalized Triangular Dynamical Systems}
    Utilizing the generalized boomerang equation, we can also estimate the $c$-boomerang uniformity of a GTDS provided that $h_i = 0$ for all $i$.
    \begin{thm}\label[thm]{Th: c-boomerang uniformity of GTDS}
        Let $\Fq$ be a finite field of characteristic $p$, let $n \geq 1$ be an integer, and let $\mathcal{F}: \Fqn \to \Fqn$ be a GTDS such that for all $1 \leq i \leq n$
        \begin{enumerate}[label=(\roman*)]
            \item $p_i (x) = x^{d_i}$ with $\gcd \left( d_i, q - 1\right) = 1$ and $p \nmid d_i$, and

            \item $h_i = 0$.
        \end{enumerate}
        Let $\mathbf{a}, \mathbf{b} \in \Fqn \setminus \{ \mathbf{0} \}$, and let $c \in \Fqx$.
        Then the $c$-boomerang uniformity table of $\mathcal{F}$ at $\mathbf{a}$ and $\mathbf{b}$ is bounded by
        \begin{equation*}
            {}_{c} \mathcal{B}_\mathcal{F} (\mathbf{a}, \mathbf{b}) \leq
            \begin{rcases}
                \begin{dcases}
                    1, & d_n = 1,\ c \neq 1, \\
                    q, & d_n = 1,\ c = 1, \\
                    q, & a_n \cdot b_n = 0, \\
                    d_n^2, &
                    \begin{cases}
                        d_n > 1,\\
                        a_n \cdot b_n \neq 0
                    \end{cases}
                \end{dcases}
            \end{rcases}
            \cdot \prod_{i = 1}^{n - 1}
            \begin{dcases}
                q, & d_i = 1, \\
                q, & a_i \cdot b_i = 0, \\
                d_i^2, &
                \begin{cases}
                    d_i > 1,\\
                    a_i \cdot b_i \neq 0.
                \end{cases}
            \end{dcases}
        \end{equation*}
    \end{thm}
    \begin{proof}
        Suppose we are given the boomerang equation system
        \begin{equation}\label{Equ: boomerang equations GTDS}
            \begin{split}
                \mathcal{F} (\mathbf{z}) &- c \cdot \mathcal{F} (\mathbf{x}) = \mathbf{b}, \\
                c \cdot \mathcal{F} (\mathbf{z} + \mathbf{a}) &- \mathcal{F} (\mathbf{x} + \mathbf{a}) = c \cdot \mathbf{b}.
            \end{split}
        \end{equation}
        Then, the last components of the boomerang equations only depend on the variable $x_n$, i.e.,
        \begin{align*}
            p_n (z_n) &- c \cdot p_n (x_n) = b_n, \\
            c \cdot p_n (z_n + a_n) &- p_n (x_n + a_n) = c \cdot b_n.
        \end{align*}
        Now we do a case distinction.
        \begin{itemize}
            \item If $\degree{p_n} = 1$, then we can apply \Cref{Lem: boomerang uniformity of linearized polynomials}, so depending on $c$ we can either have $q$ solutions or an unique solution for $(z_n, x_n)$.

            \item If $\degree{p_n} > 1$ and $a_n \cdot b_n = 0$, then the boomerang equations for $p_n$ might not be fully determined, so in this case we can have at most $q$ many solutions for $(z_n, x_n)$.

            \item If $\degree{p_n} > 1$ and $a_n \cdot b_n \neq 0$, then by \Cref{Cor: boomerang uniformity of monomials} we have at most $d_n^2$ solutions for $(z_n, x_n)$.
        \end{itemize}

        Now suppose we have a solution for the last component, say $(\hat{z}_n, \hat{x}_n) \in \Fq^2$.
        Then, we can substitute it in \Cref{Equ: boomerang equations GTDS} into the $(n - 1)\textsuperscript{th}$ component
        \begin{align*}
            f_{n - 1} (z_{n - 1}, \hat{z}_{n - 1}) &- c \cdot f_{n - 1} (x_{n - 1}, \hat{x}_n) = b_{n - 1}, \\
            c \cdot f_{n - 1} (z_{n - 1} + a_{n - 1}, \hat{z}_n + a_n) &- f_{n - 1} (x_{n - 1}, \hat{x}_n) = c \cdot b_{n - 1}.
        \end{align*}
        Since $\hat{z}_n$, $\hat{x}_n$ are field elements we can reduce these equations to
        \begin{align*}
            \alpha \cdot p_{n - 1} (z_{n - 1}) &- c \cdot \beta \cdot p_{n - 1} (x_{n - 1}) = b_{n - 1}, \\
            c \cdot \gamma \cdot p_{n - 1} (z_{n - 1} + a_{n - 1}) &- \delta \cdot p_{n - 1} (x_{n - 1} + a_{n - 1}) = c \cdot b_{n - 1},
        \end{align*}
        where $\alpha, \beta, \gamma, \delta \in \Fqx$.
        Now we have to do a case distinction on the possible cases for $a_{n - 1}, b_{n - 1}$ and $\degree{p_{n - 1}}$.
        \begin{itemize}
            \item If $\degree{p_{n - 1}} = 1$, then in principle it might happen that both equations become constant and that the left-hand sides coincide with the right-hand sides.
            Then the $c$-boomerang uniformity is maximal, i.e.\ we have to bound it by $q$.

            \item If $a_{n - 1} \cdot b_{n - 1} = 0$, then the equation system might not be fully determined.
            In that case we have to use the trivial upper bound $q$.

            \item If $a_{n - 1} \cdot b_{n - 1} \neq 0$, then we divide both equations by $a_{n - 1}^d$, the first equation by $\alpha$ and the second one by $c \cdot \gamma$.
            After a substitution analog to \Cref{Equ: substitution}, we obtain a generalized boomerang equation system with the coefficients
            \begin{equation*}
                c_1 = \frac{c \cdot \beta}{\alpha}, \qquad
                c_2 = \frac{b_{n - 1}}{\alpha \cdot a_{n - 1}^{d_{n - 1}}}, \qquad
                c_3 = \frac{\delta}{c \cdot \gamma}, \qquad
                c_4 = 1, \qquad
                c_5 = \frac{b_{n - 1}}{\gamma \cdot a_{n - 1}^{d_{n - 1}}}.
            \end{equation*}
            So by \Cref{Th: number of solutions of generalized boomerang equation} the generalized boomerang equation system $\mathcal{F} (d_{n - 1}, c_1, \allowbreak \dots, \allowbreak c_5)$ has at most $d_{n - 1}^2$ many solutions in $\Fq^2$.
        \end{itemize}
        Inductively, we now work upwards through the branches to derive the claim.
    \end{proof}

    Let $\mathbf{a}, \mathbf{b} \in \Fqn$, then we denote their Hadamard product with
    \begin{equation}
        \mathbf{a} \odot \mathbf{b} = (a_i)_{1 \leq i \leq n} \odot (b_i)_{1 \leq i \leq n} = (a_i \cdot b_i)_{1 \leq i \leq n},
    \end{equation}
    i.e.\ their pointwise product.
    Utilizing this notation we have a convenient representation of the theorem.
    \begin{cor}\label[cor]{Cor: c-boomerang probability}
        Let $\Fq$ be a finite field of characteristic $p$, let $n \geq 1$ be an integer, and let $\mathcal{F}: \Fqn \to \Fqn$ be a GTDS such that all $1 \leq i \leq n$
        \begin{enumerate}[label=(\roman*)]
            \item $p_i (x) = x^{d_i}$ with $\gcd \left( d_i, q - 1\right) = 1$ and $p \nmid d_i$, and

            \item $h_i = 0$.
        \end{enumerate}
        Let $\mathbf{a}, \mathbf{b} \in \Fqn$, and let $c \in \Fqx$.
        If for all $1 \leq i \leq n$ one has that $1 < d_i \leq d$, then
        \begin{equation*}
            {}_{c} \mathcal{B}_\mathcal{F} (\mathbf{a}, \mathbf{b}) \leq q^{n - \wt (\mathbf{a} \odot \mathbf{b})} \cdot d^{2 \cdot \wt (\mathbf{a} \odot \mathbf{b})}.
        \end{equation*}
        In particular,
        \begin{equation*}
            \prob \left[ \mathcal{F} \!: \mathbf{a} \to_{{}_{c} \mathcal{B}} \mathbf{b} \right] \leq \left( \frac{d^2}{q} \right)^{\wt (\mathbf{a} \odot \mathbf{b})}.
        \end{equation*}
    \end{cor}
    \begin{proof}
        The first claim follows from \Cref{Th: c-boomerang uniformity of GTDS}, the bound for the probability follows from the first and division by $q^n$.
    \end{proof}

    Of course, we would also like to estimate the $c$-boomerang uniformity for non-zero $h_i$'s.
    Though, this significantly complicates our estimation approach.
    For illustration, let us consider
    \begin{equation}
        \begin{split}
            f_1 (x_1, x_2) &= x_1^{d_1} \cdot g_1 (x_1) + h_1 (x_1), \\
            f_2 (x_2) &= x_2^{d_2},
        \end{split}
    \end{equation}
    where $d_1, d_2 \in \mathbb{Z}_{> 1}$ induce permutations over $\Fq$ and $h_1 \in \Fq [x]$ is non-constant.
    Let $\mathbf{a}, \mathbf{b} \in \Fq^2$, for $a_2 \neq 0$ the $c$-boomerang equations have at most $d_2^2$ many solutions.
    Further, the $c$-boomerang equations for $x_1$ become
    \begin{equation}\label{Equ: non-trivial h_i}
        \begin{split}
            z_1^{d_1} \cdot g_1 (z_2) - c \cdot x_1^{d_1} \cdot g_1 (x_2) &= b_1 + c \cdot h_1 (x_2) - h_1 (z_2), \\
            c \cdot (z_1 + a_1)^{d_1} \cdot g_1 (z_2 + a_2) &- (x_1 + a_1)^{d_1} \cdot g_1 (x_2 + a_2) \\ &= c \cdot b_1 + h_1 (x_2 + a_2) - c \cdot h_1 (z_2 + a_2).
        \end{split}
    \end{equation}
    Therefore, to estimate the number of solutions of this equation via \Cref{Th: number of solutions of generalized boomerang equation}, we in addition have to show that the polynomial systems
    \begin{equation}\label{Equ: first h_i equation}
        \begin{split}
            z_2^{d_2} - c \cdot x_2^{d_2} &= b_2, \\
            c \cdot (z_2 + a_2)^{d_2} - (x_2 + a_2)^{d_2} &= c \cdot b_2, \\
            h_1 (z_2) - c \cdot h_1 (x_2) &= b_1,
        \end{split}
    \end{equation}
    and
    \begin{equation}\label{Equ: second h_i equation}
        \begin{split}
            z_2^{d_2} - c \cdot x_2^{d_2} &= b_2, \\
            c \cdot (z_2 + a_2)^{d_2} - (x_2 + a_2)^{d_2} &= c \cdot b_2, \\
            c \cdot h_1 (z_2 + a_2) - h_1 (x_2 + a_2) &= c \cdot b_1,
        \end{split}
    \end{equation}
    do not admit any solutions in $\Fq$.
    We can view these polynomial systems as intersection of three plane curves, though a priori there is no reason why the curves coming from $h_1$ are not generated by the two other ones.
    For example for $h_1 (x) = x^{d_2}$ and $b_1 = b_2$, the $h_1$ polynomials in \Cref{Equ: first h_i equation,Equ: second h_i equation} do not add any new information at all, moreover for that parameter choice the right-hand sides of \Cref{Equ: non-trivial h_i} both become zero, so we cannot use \Cref{Th: number of solutions of generalized boomerang equation} to estimate the number of solutions.

    \subsubsection{The Boomerang Uniformity Of Two Consecutive Block Cipher Rounds}
    We finish this paper with a short discussion how to apply \Cref{Cor: c-boomerang probability} to a GTDS-based block cipher.
    Let $\Fq$ be a finite field, let $n \geq 1$ be an integer, let $\mathcal{F}: \Fqn \to \Fqn$ be a GTDS, let $\mathbf{A} \in \Fqnxn$ be an invertible matrix, and let $\mathbf{c}_i \in \Fqn$ be a constant.
    Then we define the $i$\textsuperscript{th} round function as
    \begin{equation}\label{Equ: round function}
        \begin{split}
            \mathcal{R}^{(i)} : \Fqn \times \Fqn &\to \Fqn, \\
            (\mathbf{x}, \mathbf{k}_i) &\mapsto \mathbf{A} \mathcal{F} (\mathbf{x}) + \mathbf{c}_i + \mathbf{k}_i,
        \end{split}
    \end{equation}
    where $\mathbf{k}_i$ denotes the $i$\textsuperscript{th} round key.
    For an integer $r \geq 1$ a block cipher $\mathcal{C}$ is then simply the $r$-fold composition of round functions
    \begin{equation}
        \begin{split}
            \mathcal{C}: \Fqn \times \left( \Fqn \right)^{r + 1} &\to \Fqn, \\
            (\mathbf{x}, \mathbf{k}_0, \dots, \mathbf{k}_r) &\mapsto \mathcal{R}^{(r)} \circ \cdots \circ \mathcal{R}^{(1)} (\mathbf{x} + \mathbf{k}_0),
        \end{split}
    \end{equation}
    where composition is taken with respect to the plain text variable $\mathbf{x}$.
    By \Cref{Rem: original definition} the $c$-BCT entry at $\mathbf{a}, \mathbf{b} \in \Fqn$ for two consecutive rounds of a block cipher is then given by the number of solutions of
    \begin{equation}
        {\mathcal{R}^{(i + 1)}}^{-1} \left( c^{-1} \cdot \mathcal{R}^{(i)} (\mathbf{x} + \mathbf{a}) + \mathbf{b} \right) - {\mathcal{R}^{(i + 1)}}^{-1} \left( c \cdot \mathcal{R}^{(i)} (\mathbf{x}) + \mathbf{b} \right) = \mathbf{a}.
    \end{equation}
    Performing the substitution $\mathbf{z} = {\mathcal{R}^{(i + 1)}}^{-1} \left( c \cdot \mathcal{R}^{(i)} (\mathbf{x}) + \mathbf{b} \right)$ analog to the proof of \cite[Theorem~4]{Stanica-cBoomerang} we obtain the equations
    \begin{equation}
        \begin{split}
            \mathcal{R}^{(i + 1)} (\mathbf{z}) - c \cdot \mathcal{R}^{(i)} (\mathbf{x}) &= \mathbf{b}, \\
            c \cdot \mathcal{R}^{(i + 1)} (\mathbf{z} + \mathbf{a}) - \mathcal{R}^{(i)} (\mathbf{x} + \mathbf{a}) &= c \cdot \mathbf{b}.
        \end{split}
    \end{equation}
    Substituting our definition of round functions, see \Cref{Equ: round function}, into these equations we obtain after rearranging
    \begin{equation}
        \begin{split}
            \mathcal{F} (\mathbf{z}) - \mathcal{F} (\mathbf{x}) &= \mathbf{A}^{-1} \big( \mathbf{b} + c \cdot \mathbf{c}_i - \mathbf{c}_{i + 1} + (c - 1) \cdot (\mathbf{k}_i - \mathbf{k}_{i + 1}) \big), \\
            c \cdot \mathcal{F} (\mathbf{z} + \mathbf{a}) - \mathcal{F} (\mathbf{x} + \mathbf{a}) &= \mathbf{A}^{-1} \big( c \cdot \mathbf{b} + \mathbf{c}_i - c \cdot \mathbf{c}_{i + 1} + (c - 1) \cdot (\mathbf{k}_i - \mathbf{k}_{i + 1}) \big).
        \end{split}
    \end{equation}
    In particular, if $c = 1$, $\mathbf{k}_i = \mathbf{k}$ for all $1 \leq i \leq r$ and $\mathcal{F}$ satisfies the assumptions of \Cref{Cor: c-boomerang probability}, then the corollary provides a convenient tool to estimate the boomerang uniformity of two rounds of a block cipher.

    E.g., let us consider \Hades \cite{EC:GLRRS20} without a key schedule, or its derived hash functions \Poseidon \cite{USENIX:GKRRS21} and \Poseidontwo \cite{Poseidon2}.
    Let $d, p \in \mathbb{Z}_{> 1}$ be such that $p$ is a prime and $d$ is the smallest integer such that $\gcd \left( d, p - 1 \right) = 1$, the \Hades permutation over $\Fpn$ is divided into $2 \cdot r_f$ outer ``full'' SPN rounds and $r_p$ inner ``partial'' SPN rounds
    \begin{equation}
        \Hades_{\pi} (\mathbf{x}) = \mathcal{R}_f \circ \mathcal{R}_p \circ \mathcal{R}_f (\mathbf{x}),
    \end{equation}
    where $\mathcal{R}_f$ is the composition of $r_f$ round functions instantiated with the full SPN
    \begin{equation}
        \mathcal{S}:
        \begin{pmatrix}
            x_1 \\ \vdots \\ x_n
        \end{pmatrix}
        \mapsto
        \begin{pmatrix}
            x_1^d \\ \vdots \\ x_n^d
        \end{pmatrix}
        ,
    \end{equation}
    and $\mathcal{R}_p$ is the composition of $r_p$ round functions instantiated with the partial SPN
    \begin{equation}
        \mathcal{P}:
        \begin{pmatrix}
            x_1 \\ \vdots \\ x_n
        \end{pmatrix}
        \mapsto
        \begin{pmatrix}
            x_1^d \\ x_2 \\ \vdots \\ x_n
        \end{pmatrix}
        .
    \end{equation}
    Trivially, the full and the partial SPN are GTDS instances.
    Moreover, the full SPN $\mathcal{S}$ satisfies the conditions of \Cref{Cor: c-boomerang probability}.
    Therefore, with our results one can estimate the resistance of the full SPN layer of \Hades, \Poseidon and \Poseidontwo against boomerang cryptanalysis.

    \section{Conclusion}
    In this paper we proved that the $c$-boomerang uniformity of permutation monomials $x^d$ and $x^\frac{1}{d}$, where $d$ is not divisible by the characteristic $p$, is bounded by $d^2$.
    We also applied this bound to estimate the $c$-BCT table entries for a class of vector-valued cryptographic permutations, among them the well-known Substitution-Permutation Network.
    In the future, we hope to see refinements of our methods that allow to estimate the $c$-BCT entries of a full GTDS.


    %
    %

    \bibliographystyle{amsplain}
    \bibliography{abbrev0.bib,crypto.bib,literature.bib}

\end{document}